\documentclass{amsart}
\usepackage{amsmath,amsthm,amscd,amssymb, color}
\usepackage{thmtools}
\usepackage[colorlinks, citecolor = blue,backref]{hyperref}
\usepackage[alphabetic]{amsrefs}
\usepackage[capitalize, nameinlink]{cleveref}
\usepackage{enumerate}
\newcommand{\Q}{\mathbb Q}

\newcommand{\Z}{\mathbb Z}

\newcommand{\eps}{\varepsilon}

\newcommand{\calO}{\mathcal O}

\newcommand{\A}{\mathbb A}

\newcommand{\bmx}{\left( \begin{matrix}}
\newcommand{\emx}{\end{matrix} \right)}

\newcommand{\new}{\mathrm{new}}
\newcommand{\odd}{\mathrm{odd}}
\newcommand{\even}{\mathrm{even}}
\renewcommand{\mod}{\bmod}

\usepackage{bbm}

\newcommand{\leg}{\overwithdelims ()}

\DeclareMathOperator{\GL}{GL}

\DeclareMathOperator{\tr}{tr}

\newtheorem{lem}{Lemma}
\numberwithin{lem}{section}
\newtheorem{prop}[lem]{Proposition}
\newtheorem{thm}[lem]{Theorem}
\newtheorem{cor}[lem]{Corollary}

\crefname{question}{Question}{Questions}

\crefname{conja}{Conjecture}{Conjectures}

\theoremstyle{remark}

\theoremstyle{definition}
\newtheorem{ex}[lem]{Example}

\numberwithin{equation}{section}

\begin{document}

\title{Root number bias for newforms}
\author{Kimball Martin}
\email{kimball.martin@ou.edu}
\address{Department of Mathematics, University of Oklahoma, Norman, OK 73019 USA}
\thanks{This work was supported by the Simons Foundation
(Collaboration Grant 512927), the Japan Society for the 
Promotion of Science (Invitational Fellowship L22540), and the
Osaka Central Advanced Mathematical Institute (MEXT Joint Usage/Research Center on Mathematics and Theoretical Physics JPMXP0619217849).}

\maketitle

Corrected version (\today)\footnote{After this paper appeared in \emph{Proc.\ Amer.\ Math.\ Soc.}\ (2023), I discovered a mathematical misprint in Theorem 1.2.  This version corrects that error.  (See footnote for Theorem 1.2.)}

\begin{abstract}
Previously we observed that newforms obey a strict bias towards root number $+1$ in squarefree levels: at least half of the newforms in
$S_k(\Gamma_0(N))$ with root number $+1$ for $N$ squarefree,
and it is strictly more than half outside of a few special cases.  
Subsequently, other authors treated levels which are cubes of 
squarefree numbers.  Here we treat arbitrary levels, and find that
if the level is not the square of a squarefree number, this 
strict bias still holds for any weight.  In fact the number of such 
exceptional levels is finite for fixed weight, and 0 if $k < 12$. 
We also investigate some variants of this question to better 
understand the exceptional levels.
\end{abstract}


\section{Introduction}

Throughout, $N$ and $k$ are positive integers with $k$ even. 
Let $S_k(N) = S_k(\Gamma_0(N))$ be the space of weight $k$,
level $N$ cusp forms, and $S_k^\new(N)$ be the 
new subspace.  The root number of a newform 
$f \in S_k^\new(N)$ is the sign in the functional equation for
$L(s,f)$, and equals $(-1)^{k/2}$ times the 
eigenvalue of the Fricke involution $W_N = \prod_{p | N} W_p$.
It is known that, on average, 50\% of newforms have root number $+1$
and 50\% of newforms have root number $-1$.  We examine
the distribution of root numbers more precisely.

Let $S_k^\new(N)^\pm$ denote the subspace spanned
by newforms with root number $\pm 1$.  Set
\[ \Delta(N,k) = \dim S_k^\new(N)^+ - \dim S_k^\new(N)^-. \] 
In \cite{me:refdim}, we observed that the trace formula for
$W_N$ implies that newforms obey a strict bias towards
root number $+1$ for squarefree levels $N$, in the sense that
$\Delta(N,k) \ge 0$ for all $k$ and squarefree $N$.
Further, $\Delta(N,k)$ is an elementary expression in the class number
of $\Q(\sqrt{-N})$, and is strictly positive unless  (i) $N = 2, 3$
and $k$ obeys a certain congruence, or
(ii) $k=2$ and $\dim S_2^\new(N) = 0$ or
$N = 37, 58$.

Subsquently, these results were extended to the case of
cubes of squarefree levels by \cite{PQ} ($k > 2$) and 
\cite{LPW} ($k \ge 2$, which also considers Hilbert
modular forms).  The methods in those papers are 
much more involved, and use Petersson and Jacquet--Zagier
trace formulas, respectively.
Here we explicitly compute the trace of $W_N$ on 
$S_k^\new(N)$ for arbitrary $(N,k)$, which yields the following.

\begin{thm} \label{thm1}
\begin{enumerate}
\item If $N$ is not the square of a squarefree number, or if $k \le 10$
or $k= 14$,  then 
there is a strict bias of newforms toward root number $+1$,
i.e., $\Delta(N,k) \ge 0$.

\item Suppose $N$ is the square of a squarefree number, i.e., a cubefree
square.  For sufficiently large $k$, $\Delta(N, k) < 0$ if and only if
$(-1)^{k/2} = - \mu(\sqrt N) = -\prod_{p | N} (-1)$.

\item For a fixed $k$, there are only finitely many $N$ such that
$\Delta(N,k) < 0$.
\end{enumerate}
\end{thm}

Moreover, in (1), $\Delta(N,k)$ is typically strictly positive.
Precise conditions are given in \cref{cor:main}.  We will discuss 
possible reasons 
for the curious difference in behaviour for cubefree square levels in \cref{sec:12}.

It is not too hard to make statements (2) and (3) effective.  We
do not explicate this, but refer the interested reader to the proof of
\cref{cor:main}(2).  

As in the squarefree case, we in fact get an elementary expression for
the exact size of the bias, and it is essentially the class number
of $\Q(\sqrt{-N})$.  The precise formula for $\Delta(N,k)$
breaks up into several cases according to whether $N$ is 1, 2, 3 or 4
times the square of a squarefree number, or none of these
(the generic case). 
Here we just describe the formula in the generic case,
and refer to \cref{prop:trnew} for all cases.

Write $N = N_1 N_2^2$, where $N_1$ is squarefree.  
Define $\beta(N) \in \{ 1, 2, 3, 4 \}$ by
\begin{equation} \label{eq:beta}
\beta(N) = 
\begin{cases}
1 & \text{if } N_1 \equiv 1,2 \mod 4, \text{ or } 2 \parallel N_2; \\
2 &\text{if } N_1 \equiv 3 \mod 4 \quad \text{and } 4\mid  N_2; \\
3 - {-N_1 \leg 2} & \text{if } N_1 \equiv 3 \mod 4 \quad \text{and }  2 \nmid N_2.
\end{cases}
\end{equation}

For a fundamental discriminant $-D < 0$, let $h'(-D)$ be the 
``unit group weighted'' class number, i.e., one half of the number of integral
units in $\Q(\sqrt {-D})$ times the usual class number $h(-D)$.  
Precisely, set $h'(-4) = \frac 12$
and  $h'(-3) = \frac 13$, and take $h'(-D) = h_{\Q(\sqrt{-D})}$ to be
usual class number for $D > 4$.

\begin{thm} \label{thm2}
Suppose $N$ is not $1$, $2$, $3$ or $4$ times the
square of a squarefree number.   Let $-D \in \{ -N_1, -4N_1 \}$
be the discriminant of $\Q(\sqrt{-N})$. Then
 \[ \Delta(N,k) = \frac 12 \beta(N) \prod_{p | N_2} \left( \phi (p^{v_p(N_2)}) -  
\phi (p^{v_p(N_2)-1)}){-D \leg p} \right) h'(-D) - \delta,   \]
where $\phi$ is the Euler phi function,  $\delta = 1$ if $(N_2,k) = (1,2)$ and 
$\delta = 0$ otherwise.\footnote{The published version mistakenly stated 
``$\delta = 1$ if $(N,k) = (1,2)$ and $\delta = 0$ otherwise.''  Note that this theorem comes from \cref{prop:trnew}(1) below, and the condition for $\delta = 1$ was stated correctly there in the published version.}

In particular, we always have $\Delta(N,k) \ge 0$.  Moreover, 
$\Delta(N,k) = 0$ if and only if  
(i) $k=2$ and either $\dim S_k^\new(N) = 0$ or $N \in \{37, 58 \}$; or
(ii) $2 \parallel  N_2$ and $N_1 \equiv 7 \mod 8$.
\end{thm}

Note that if $v_p(N)$ is odd for all $p \mid N$, then ${-D \leg p} = 0$
for all $p \mid N_2$, and thus the main term in the above formula
is just $\frac 12 \beta(N) \phi(N_2) h'(-D)$.  In particular, this
recovers the formulas in \cite{me:refdim} and \cite{PQ}.

\subsection{Remarks on proof}
The proof has in essence two steps.  
The quantity $\Delta(N,k)$ equals the trace of $(-1)^{k/2} W_N$
on $S_k^\new(N)$.
First, we use a trace formula due to Yamauchi \cite{yamauchi} and 
Skoruppa--Zagier \cite{SZ}
which expresses the trace of $W_N$ on $S_k(N)$
as an alternating sum of class numbers.  
Using class number
relations and elementary but careful analysis, we rewrite
this trace in terms of a single weighted class number $h'(-D)$
(see \cref{prop:trold}, which can also be derived from the recent work
\cite{ZZ}).
Second, one can express the trace of $W_N$ on the new subspace
as an alternating sum over $M^2 \mid N$ of traces on $S_k(N/M^2)$.
Once again, we use class number relations and elementary analysis
to obtain our formula for $\Delta(N,k)$ (see \cref{prop:trnew}).

We remark that our approach is much more straightforward
and simple than the approaches taken in \cite{PQ} and \cite{LPW}.
In \cite{me:refdim}, we restricted to squarefree levels because
our main focus in that paper was dimension formulas with prescribed
local signs at multiple ramified places, and that problem seems
considerably more complicated for non-squarefree levels.  The result 
therein about root number bias for squarefree levels was merely
a curious observation we made along the way, which was already 
immediate from the Yamauchi/Skoruppa--Zagier trace formulas 
(though, to our knowledge,
it had not been noticed before).  For squarefree levels, both steps
in our present proof are trivial because (i) the trace formula for 
$W_N$ on $S_k(N)$ is very simple for squarefree $N$
 and (ii) there is no oldform contribution
to this trace that we need to subtract off.

\subsection{Further questions} \label{sec:12}
\cref{thm1} naturally prompts 2 questions: (i) what is the
reason for this bias, and (ii) what is the reason for these exceptions.

We briefly speculated on (i) in \cite{me:refdim}, and mentioned
two ideas that at least support the existence of a bias towards
root number $+1$.  First, the ``$L$-functions from nothing''
perspective suggests that $L$-functions in small weight and level
tend to have root number $+1$ (e.g., see \cite[p.\ 10]{me:refdim}).  
Second, when $k=2$, at least
for prime levels, $\Delta(p,2) \ge 0$ is forced upon us by
the Jacquet--Langlands correspondence together with the
fact that the type number is at least $\frac 12$ times the
class number for maximal orders in quaternion algebras
of prime discriminant. 

Anna Medvedovsky has since suggested to me that one can
also reinterpret the $k=2$ case for squarefree $N$ 
by comparing the genus of $X_0(N)$ with its quotient by the 
Fricke involution.   
One can note that the formulas for $\Delta(N,k)$ are essentially 
independent of $k$, and thus are controlled by the $k=2$ case.  
One heuristic reason for this latter fact is that the main term in
dimension formulas of spaces of modular forms is a product
of local factors, including one at infinity for the weight $k$. 

Thus there are at least some hints for the above root number bias
that do not only rely on the trace formula, but 
I do not have a compelling existential explanation for this bias.
On the other hand, if one just counts elliptic curves up to isogeny by conductor, i.e., weight 2 rational newforms, examining Cremona's database (see \cite{LMFDB}) shows there is at least an initial 
bias towards root number $-1$.

Moving on to (ii), let me first explain how the exceptions to the
strict bias arise in the proof.  First, the trace of $(-1)^{k/2} W_N$
on $S_k(N)$ is 
non-negative (see \cref{prop:trold}).  The only
question now is whether this still holds when we subtract off
the oldform contribution, i.e., if 
$(-1)^{k/2} \sum_{M^2 | N} \mu(M) \tr W_{N/M^2} \ge 0$.  
For a fixed $N$, the
the terms for each $M^2 \mid N$ are all similar looking 
expressions in $h'(-D)$ that can be collected together,
except in the special case that $N$ is a perfect square.  
In that case, the term with $M = \sqrt N$ is $\mu(M)$ times the trace of
$W_{N/M^2} = W_1 = \mathrm{Id}$ 
on $S_k(1)$, i.e., $\mu(M) \dim S_k(1)$.  
Since $\dim S_k(1) \simeq \frac k{12}$, we find that
$\Delta(N,k) \approx \frac 14 \phi(M) + (-1)^{k/2} \mu(M) \frac k{12}$.
This yields parts (2) and (3) of \cref{thm1}.

One thing that is special about levels which are perfect squares
is that these are precisely the levels where the newspaces
are not generated by theta series attached to definite
quaternion algebras (see \cite{HPS} and \cite{me:basis}).
In particular if the level $N = M^2$, where $M$ is squarefree,
then $S_k^\new(N)$ is generated by theta series together with
twists of forms of level 1 and level $M$ (now including forms with
nebentypus).


The discussion of the proof above, combined with this consideration
about theta series, suggests that the reason
the strict bias towards root number $+1$ does not persist
in (certain) square levels may be due to newforms which are twists from full
level.  

In \cref{sec:p2}, we describe what happens if one removes
the contribution from twists of level 1 forms.
Specficially, denote by $S_k^\new(N)'$ be the
subspace of $S_k^\new(N)$ generated by newforms which are
not twists from level $1$.  Let $\Delta(N,k)'$ be the number of
newforms in $S_k^\new(N)'$ with root number $+1$ minus the
number of newforms with root number $-1$.

\begin{prop} \label{prop:13}
Let $M > 1$ be squarefree.  The spaces $S_k^\new(M^2)'$ have a
strict bias towards root number $+1$, i.e., $\Delta(M^2,k)' \ge 0$, 
for all $k$ if and only if 
$M$ is odd and $M$ has an odd number of prime factors which
are $3 \mod 4$.  Otherwise, for sufficiently large $k$,
the sign of $\Delta(M^2,k)'$ is $(-1)^{k/2} \mu(M)$.
\end{prop}

Thus even though the proof of our theorems suggests the difference
in root number bias behavior for cubefree square levels comes from
level 1 forms, this difference does not always disappear when
we omit twists of level 1 forms.  In fact, when $M$ is even and 
squarefree, there are no twists from level 1. 

From this perspective, it is natural to ask
what happens if one excludes all newforms which are twists from
smaller levels, i.e., if one restricts to minimal newforms.
Since every newform in $S_k^\new(4)$ is minimal, and $\Delta(4,k) < 0$
whenever $k \equiv 0 \mod 4$ and $k  \ge 12$, the sum of root numbers
over minimal newforms is not always positive.
We also exhibit examples of $p \equiv 3 \mod 4$ such that
$S_k^{\new}(p^2)$ has more minimal newforms with root number 
$-1$ than $+1$. 
On the other hand, we give a simple local condition
 which forces root numbers to be
 perfectly equidistributed for minimal forms in certain (not necessarily
 square) levels.

\begin{prop} \label{prop:14}
 Let $N \ge 1$, and let $S_k^{\min}(N)$  be the
subspace of $S_k^\new(N)$ generated by newforms  which are not
twists from smaller levels.  Suppose there exists $p^2 \parallel N$
such that ${-N/p^2 \leg p} = 1$.  Then $\dim S_k^{\min}(N)^+
= \dim S_k^{\min}(N)^-$, i.e., root numbers are perfectly balanced
for minimal newforms.
\end{prop}

The congruence conditions on primes dividing $N$ 
in these propositions arise from the way that local root 
numbers behave under twisting by quadratic characters.

\subsection{Further remarks}
As in \cite{me:refdim}, our formula for $\Delta(N,k)$ implies
an exact formula for the dimensions 
\[ \dim S^\new_k(N)^\pm = \frac 12
 \left( \dim S^\new_k(N) \pm \Delta(N,k) \right). \]
We do not write down these dimension formulas explicitly,
but they immediately follow by comparing
with the explicit formula for $\dim S_k^\new(N)$ in \cite{martin}.

Consequently, our formula for $\Delta(N,k)$ implies
that root numbers of are $+1$ (or $-1$) for 50\% of newforms as 
$k + N \to \infty$.  We are not aware of an explicit proof of this
fact in the literature for non-squarefree $N$, though proofs were
surely known to experts.   In any case, our formulas together
with standard class number bounds yield equidistribution of root 
numbers with a very good error estimate.

The original motivation for the dimension formulas for
newforms with prescribed Atkin--Lehner signs in \cite{me:refdim}
was to use them to obtain mod 2 congruences in \cite{me:cong2}.
Namely, we showed that perfect equidistribution of Atkin--Lehner
sign patterns in squarefree levels essentially means that every newform
is mod 2 congruent to one with any desired Atkin--Lehner signs at those
places.  It may be interesting to see whether the cases of perfect 
equidistribution of root numbers in non-squarefree levels we give here 
are similarly related to mod 2 congruences.

The work \cite{LPW} estabishes biases of root numbers for Hilbert modular
forms for certain levels and base fields.  One should be able to imitate
\cite{SZ} to derive a trace formula for Fricke, as well as Atkin--Lehner,
involutions of Hilbert modular forms.  Then we expect that the strategy used 
in the present paper (resp.\ the one in \cite{me:refdim}) 
should yield distributions of root numbers (resp.\ Atkin--Lehner sign patterns)
in quite general settings (resp.\ for squarefree levels).

As a check on our formulas, we compared them numerically 
with newform and root number calculations in Sage
\cite{sage} for various ranges of values of $(N,k)$.

\subsection*{Acknowledgements}
I was inspired to revisit the question of root number biases
during a visit to MIT in Spring 2022
after a series of individual discussions with Andrew Knightly,
Anna Medvedovsky, and Andrew Sutherland.  I also thank them
for subsequent comments.  I am grateful to the referee for a
careful reading, and thoughtful comments and corrections.  


\section{Preliminaries}

\subsection{Notation}
Throughout, $M$ and $N$ denote positive integers, and $k \ge 2$
is even.  Let $v_p(N) = \max \{ r \in \Z : p^r \mid N \}$.
For $N \in \mathbb N$, denote by $N^\odd$ the odd part of $N$,
i.e., $N = 2^{v_2(N)} N^\odd$.  Let $\square \subset \mathbb N$
be the multiplicative submonoid of squares.  Let $\mu$ be the
M\"obius function.

Denote by $\tr W_N$ the trace of $W_N$ on the
full cusp space $S_k(N)$, and by $\tr W_N^\new$ the trace on the
new subspace $S_k^\new(N)$.  These traces also depend on the
weight $k$, but we suppress it in our notation.

For a statement $*$,
we use $\delta_*$ to mean 1 if the statement $*$ is satisfied,
and $0$ otherwise.  E.g., $\delta_{k=2}$ represents the
Kronecker delta function $\delta_{k,2}$.

The following function arises in the trace formula for $W_N$.
For $s \in \Z_{\ge 0}$, 
set $p_k(s) = \frac{\rho^{k-1} - \bar \rho^{k-1}}{\rho - \bar \rho}$
where $\rho, \bar \rho$ are the roots of $X^2-sX + 1$ if $s^2 \ne 4$,
and set $p_k(2) = k-1$.

\subsection{Class number relations} \label{sec:clnos}
Let $-D < 0$ be a discriminant, and $h(-D)$ be the class number
of the quadratic order $\calO_D$ of discriminant $-D$.  Let $h'(-D)$ be
the class number weighted by $[\calO_D^\times : \Z^\times]^{-1}$, i.e.,
$h'(-D) = h(-D)$ unless $D < 5$, and then $h'(-4) = \frac 12$ and $h'(-3) =
\frac 13$.
Define the Hurwitz class number $H(-D)$ to be 
the number of positive definite binary quadratic
forms of discrimimant $\Delta$ up to equivalence, where we
weight forms of discriminant $-4$ and $-3$ by $\frac 12$ and $\frac 13$,
respectively.  (The class numbers $h(-D)$ and $h'(-D)$ can be interpreted
in terms of counting primitive forms.)  We also set $H(0) = - \frac 1{12}$.

Now suppose $-D$ is a fundamental discriminant, and $\lambda > 0$.
Then 
\begin{equation} \label{eq:hp-nfd}
h'(-\lambda^2 D) =  \lambda \prod_{p | \lambda} 
( 1 - {-D \leg p} \frac 1p) \, h'(-D) = \lambda \sum_{t | \lambda}
\mu(t) {-D \leg t} \frac 1t  \,  h'(-D).
\end{equation}
(E.g., see \cite[Corollary 15.40]{cohn}.)

We can write the Hurwitz class number in terms of $h'$ by
\begin{equation} \label{eq:Hdef}
H(-\lambda^2 D) = \sum_{t | \lambda} h'(-D t^2).
\end{equation}
Then
\begin{equation} \label{eq:H}
 H(-\lambda^2 D) = \sum_{m | \lambda} \sum_{t | m} 
\mu(t) {-D \leg t} \frac mt h'(-D) 
= \sum_{t | \lambda} \mu(t) {-D \leg t}  {\sigma(\lambda/t)} h'(-D).
\end{equation}


\section{Traces on newspace}


Since, $\Delta(N,k) = (-1)^{k/2}\tr W_N^\new$ we want to
compute $\tr W_N^\new$.  A formula for the trace of
a product $T_n W_M$ of Hecke and Atkin--Lehner operators 
on $S_k(N)$, together with a relation for how the trace on the
newspace is related to traces on full cusp spaces was given 
by Yamauchi \cite{yamauchi}.  Unfortunately, Yamauchi's
paper contains some clerical errors.  A corrected formulation
of these traces was later given by Skoruppa and Zagier 
\cite{SZ}, and we will use their formulation. 

We just need the case that $n=1$ and $M=N$, but
some effort is still required to get from the Skoruppa--Zagier
formula to a relatively simple expression for
$\Delta(N,k)$.  We will first find a simple expression for 
$\tr W_N$, and then use it to get our desired formula for $\tr W_N^\new$.

As explained in \cite[pp.\ 132--133]{SZ},
a newform $f \in S_k^\new(M)$ with root number $(-1)^{k/2}w_f$
contributes $w_f$ to $\tr W_N$ if $N/M \in \square$ and
$0$ otherwise.  Thus
\[
\tr W_N = \sum_{M \mid N, \, N/M \in \square} \tr W_M^\new,
\]
and so by M\"obius inversion,
\begin{equation} \label{eq:trnew1}
\tr W_N^\new = \sum_{Q^2 \mid N} \mu(Q) \tr W_{N/Q^2}.
\end{equation}

Now a special case of \cite[(2.7)]{SZ} yields
\begin{equation} \label{eq:SZ}
\tr W_N = -\frac 12 \sum_{M | N, N/M \in \square} \mu(\sqrt{N/M})
\sum_s I_s(k,M) - \frac 12 \delta_{N \in \{ 1, 4 \}} + \delta_{k=2},
\end{equation}
where
\[ I_s(k,M) = p_k(s/\sqrt M) H(s^2-4M) ,\]
and $0 \le s \le 2 \sqrt M$ such that $\sqrt{MN} \mid s$.
The only way that $s > 0$ is possible then is if $N \le 4$.
Note that $p_k(0) = (-1)^{(k-2)/2}$, so when
$s = 0$ we get
\[ I_0(k,M) =  (-1)^{(k-2)/2} H(-4M). \]

\subsection{Small levels} \label{sec:smallN}
We first deal with levels $N \le 4$.

The Fricke involution is trivial if $N=1$. In particular we have
\[ \tr W_1 = \dim S_k(1) = 
\begin{cases}
\lfloor \frac k{12} \rfloor & k \not \equiv 2 \mod 12 \\
\lfloor \frac k{12} \rfloor - 1 + \delta_{k=2}& k \equiv 2 \mod 12. \\
\end{cases} \]

The cases $N=2, 3$ are included in the case of $N$ squarefree
(e.g., see \cite[Thm 2.2]{me:refdim}).  Explicitly, we have
\[ \tr W_2 = \tr W_2^\new = 
\begin{cases}
(-1)^{k/2}(1 - \delta_{k=2}) & k \equiv 0, 2 \mod 8 \\
0 & k \equiv 4, 6 \mod 8,
\end{cases} \]
and
\[ \tr W_3 = \tr W_3^\new = 
\begin{cases}
(-1)^{k/2}(1 - \delta_{k=2}) & k \equiv 0, 2, 6, 8 \mod 12 \\
0 & k \equiv 4, 10 \mod 12.
\end{cases} \]

If $N = 4$, we have
\[ \tr W_4 = \frac 12 \left( I_0(k,1) + I_2(k,1) -  I_0(k,4) - I_4(k,4) \right)
- \frac 12 + \delta_{k=2} \]
Note that $I_0(k,1) - I_0(k,4) = (-1)^{(k-2)/2} (H(-4) - H(-16)) = (-1)^{k/2}$,
and $I_2(k,1) = p_k(2) H(0) = I_4(k,4)$.  Thus
\[ \tr W_4 = 
\begin{cases}
0 & k \equiv 0 \mod 4 \\
-1 + \delta_{k=2} & k \equiv 2 \mod 4.
\end{cases} \]
Then
\[ \tr W_4^\new = \tr W_4 - \tr W_1 = 
\begin{cases}
-\lfloor \frac k{12} \rfloor & k \equiv 0, 2, 4, 8 \mod 12 \\
-\lfloor \frac k{12} \rfloor - 1 & k \equiv 6, 10 \mod 12.
\end{cases} \]

\subsection{Generic levels: trace on full cusp space}
Now suppose $N > 4$.  Then the only $s$-terms in \eqref{eq:SZ}
are $s=0$, and we have
\[ \tr W_N = (-1)^{k/2} \frac 12 \sum_{M | N, N/M \in \square} \mu(\sqrt{N/M})
H(-4M)  + \delta_{k=2}. \]

We can uniquely write $N=N_1 N_2^2$ where $N_1$ is squarefree. 
If $N/M$ is a square, then we can write $M = N_1 M_2^2$
where $M_2 \mid N_2$.  Hence
\[  \tr W_N = (-1)^{k/2} \frac 12 \sum_{M_2 | N_2} \mu(N_2/M_2) H(-4N_1 M_2^2)
+ \delta_{k=2}. \]

\begin{prop} \label{prop:trold}
Suppose $N > 4$, and write $N= N_1 N_2^2$ with $N_1$
squarefree as above.  Then
\[ (-1)^{k/2}\tr W_N + \delta_{k=2} =
\begin{cases}
\frac 12 h'(-4N) &\text{if } N_1 \equiv 1, 2 \mod 4, \\
\frac 12 \left(3 - {-N_1 \leg 2}\right) h'(-N) &\text{if } N_1 \equiv 0, 3 \mod 4, \quad N_2 \text{ odd}, \\
h'(-N) &\text{if } N_1 \equiv 0, 3 \mod 4, \quad N_2 \text{ even}. \\
\end{cases} \]
\end{prop}

See \cite{ZZ} for an alternative proof, which also allows for quadratic nebentypus.

\begin{proof}

First suppose $-4N_1$ is a fundamental discriminant, i.e., $-N_1 \equiv 2, 3 \mod 4$.  We compute that
\begin{align*}
(-1)^{k/2} \tr W_N + \delta_{k=2} &= 
\frac 12 \sum_{M_2 | N_2} \sum_{t | M_2} \mu(N_2/M_2) 
\sigma(M_2/t) \mu(t) {-4N_1 \leg t} h'(-4N_1) \\
&= \frac 12 \sum_{t | N_2} \frac{N_2}t \mu(t) {-4N_1 \leg t} h'(-4N_1) \\
\nonumber &= \frac 12 h'(-4N).
\end{align*}
Here we got from  
the first equation to the second
by interchanging
the order of summation and observing that for fixed $t \mid N_2$, 
\[ \sum_{t | M_2 | N_2}  \mu(N_2/M_2) \sigma(M_2/t) 
= \sum_{M_2' | (N_2/t)} \mu((N_2/t)/M_2') \sigma(M_2'), 
\]
where we have written $M_2 = tM_2'$.  Then observe that this
expression is just the Dirichlet convolution $(\mu * \sigma)(N_2/t)
= \mathrm{Id}(N_2/t) = N_2/t$.  We will use this latter fact again
in the remaining cases.

For the rest of the proof, assume $-N_1$ is a fundamental discriminant.
Then
\begin{equation} \label{eq:p1b}
(-1)^{k/2} \tr W_N + \delta_{k=2} = 
\frac 12 \sum_{M_2 | N_2} \sum_{t | 2M_2} \mu(N_2/M_2) 
\sigma(2M_2/t) \mu(t) {-N_1 \leg t} h'(-N_1).
\end{equation}

First consider the case that $N_2$ is odd. 
Note that for $t$ odd and $M_2 \mid N_2$, 
$\sigma(2M_2/t) = 3 \sigma(M_2/t)$.  Consequently, 
\begin{multline*}
 \sum_{M_2 | N_2} \sum_{t | 2M_2} \mu(N_2/M_2) \sigma(2M_2/t) \mu(t) {-N_1 \leg t} \\
 = \sum_{t | M_2} \sum_{M_2 | N_2}  \mu(N_2/M_2)
\left( \sigma(2M_2/t) \mu(t) {-N_1 \leg t}  + \sigma(M_2/t) \mu(2t) {-N_1 \leg 2t} \right)  \\
= \sum_{t | N_2} \sum_{t | M_2 | N_2}  \mu(N_2/M_2)
\sigma(M_2/t) \mu(t) {-N_1 \leg t} \left( 3  - {-N_1 \leg 2} \right)  \\
= \sum_{t | N_2}(\mu * \sigma)(N_2/t) \mu(t) {-N_1 \leg t} \left( 3  - {-N_1 \leg 2} \right)
 \\
= \sum_{t | N_2} \frac{N_2}t {-N_1 \leg t} \mu(t) \left( 3  - {-N_1 \leg 2} \right).
\end{multline*}
Plugging this into \eqref{eq:p1b} and using \eqref{eq:hp-nfd} yields the second case of the proposition.

Finally suppose that $N_2$ is even, so we can write $N_2 = 2^e N_2^\odd$
where $e \ge 1$.  For an $M_2 \mid N_2$ appearing in the double sum in
\eqref{eq:p1b}, we must have $M_2 = 2^f M_2^\odd$ where $f \in \{ e - 1, e \}$
and $M_2^\odd \mid N_2^\odd$ due to the factor $\mu(N_2/M_2)$.  Furthermore
the $\mu(t)$ factor means that we can restrict to $t \mid 2M_2^\odd$ in the same
double sum.  Hence \eqref{eq:p1b} is equal to $\frac 12 h'(-N_1)$ times

\begin{multline*}
\sum_{M_2^\odd | N_2^\odd} \sum_{t | 2M_2^\odd} 
\left( \mu(2N_2^\odd/M_2^\odd) \sigma(2^{e} M_2^\odd/t) \mu(t) { - N_1 \leg t} + \right. \\
\left.
\mu(N_2^\odd/M_2^\odd) \sigma(2^{e+1}M_2^\odd/t) \mu(t) { - N_1 \leg t} \right) = \\
\sum_{M_2^\odd | N_2^\odd} \sum_{t | 2M_2^\odd} 
\mu(N_2^\odd/M_2^\odd) \left(\sigma(2^{e+1}M_2^\odd/t) - \sigma(2^{e} M_2^\odd/t)\right)  \mu(t) { - N_1 \leg t}.
\end{multline*}
Note that $\sigma(2^j m) = (2^{j+1}-1) \sigma(m)$ for $m > 0$ odd and $j \ge 0$.  Now considering the $t$ odd and $t$ even terms
separately, we rewrite the above as
\begin{align*} 2^{e} \sum_{M_2^\odd | N_2^\odd} \sum_{t | M_2^\odd} 
\mu(N_2^\odd/M_2^\odd)\sigma(M_2^\odd/t) \mu(t) { - N_1 \leg t} 
\left( 2 - {-N_1 \leg 2} \right) &= \\
2^{e}\left( 2 - {-N_1 \leg 2} \right) \sum_{t | N_2^\odd}
(\mu * \sigma)(N_2^\odd/t) \mu(t) { - N_1 \leg t} &= \\
2^{e}\left( 2 - {-N_1 \leg 2} \right) \sum_{t | N_2^\odd}
\frac{N_2^\odd}t \mu(t) { - N_1 \leg t}.
\end{align*}
Multiplying this expression by 
$\frac 12 h'(-N_1)$ and applying both equalities in
\eqref{eq:hp-nfd} yields
\[ 2^{e-1} \left( 2 - {-N_1 \leg 2} \right) h'(-N_1 (N_2^\odd)^2)
= h'(-N). \]
This completes the last case of the proposition.
\end{proof}

\subsection{Alternating class number sums}
Let $N \ge 1$ and write $N = N_1 N_2^2 > 4$ with $N_1$
squarefree.
Then by \eqref{eq:trnew1}, we have
\begin{equation} \label{eq:trnew2}
 \tr W_N^\new = \sum_{Q | N_2} \mu(N_2/Q) \tr W_{N_1 Q^2}.
\end{equation}
Comparing this with \cref{prop:trold}, to get a simplified formula
for $\tr W_N^\new$, we essentially just need to evaluate
\begin{equation} \label{eq:altsum}
 \sum_{Q | N_2} \mu(N_2/Q) h'(-DQ^2),
\end{equation}
where $-D = -4N_1$ or $-D = -N_1$, whichever is a fundamental 
discriminant.  We will also need to include ``correction terms'' 
when $k=2$ or when $N_1 Q^2 \le 4$, and distinguish the $Q$
odd from $Q$ even cases when $-D = -N_1$,
but let us first evaluate
\eqref{eq:altsum}.

We can rewrite the sum in \eqref{eq:altsum} as
$\sum_{Q | N_2} \sum_{t | Q} \mu(N_2/Q) \frac Qt \mu(t) {-D \leg t}
h'(-D)$, which is equal to
$\sum_{t | N_2} \left( \sum_{Q' | N_2/t} \mu((N_2/t)/Q') Q' \right)
\mu(t){-D \leg t}h'(-D).$
The inner sum is the Dirichlet convolution $(\mu * \mathrm{Id})(N_2/t) 
= \phi(N_2/t)$.  Hence 
\begin{equation} \label{eq:altsum2}
 \sum_{Q | N_2} \mu(N_2/Q) h'(-DQ^2) = c(-D, N_2) h'(-D),
\end{equation}
where
\begin{align} \nonumber
c(-D, n) &= \sum_{t | n} \phi(n/t) \mu(t){-D \leg t} \\
&= \prod_{p || n} \left(p-1-{-D \leg p}\right) \prod_{p^2 | n} p^{v_p(n)-2}(p-1)
\left(p - {-D \leg p} \right). \label{eq:cprodform}
\end{align}
The second equality follows from observing $c(-D, n)$ is a Dirichlet convolution
of multiplicative functions in $n$, and thus multiplicative in $n$, and
then noting that $c(-D, p^r) = \phi(p^r) - \phi(p^{r-1}) {-D \leg p}$.

We make two remarks.  First, $c(-D, N_2) \ge 0$ for all $N$,
and $c(-D, N_2) = 0$ if and only if  $2 \parallel N_2$ and ${-D \leg 2} = 1$.
Second, if $N_2 \mid D$ (which, when $-D$ is odd, is equivalent to
$v_p(N) \ne 2$ for all $p$), then we simply have that 
$c(-D, N_2) = \phi(N_2)$.

\subsection{Main and correction terms} \label{sec:correct}
As above, write $N = N_1 N_2^2$ ($N \ge 1$) with $N_1$ 
squarefree, and let $-D$ be the discriminant of $\Q(\sqrt{-N})$. 
Set
\[ b(N) = \begin{cases}
\frac 12 & -N_1 \equiv 2, 3 \mod 4, \\
\frac 12 \left(3 - {-N_1 \leg 2}\right)  & -N_1 \equiv 1 \mod 4, \quad N_2 \text{ odd}, \\
1 & -N_1 \equiv 1 \mod 4, \quad N_2 \text{ even}. \\
\end{cases} \]
Then \cref{prop:trold} asserts that 
\begin{equation}
 \tr W_{N_1 Q^2} = (-1)^{k/2} b(N_1 Q^2) h'(-DQ^2) + \delta_{k=2},
\end{equation}
when $N_1 Q^2 > 4$.  Combining this with \eqref{eq:trnew2} implies
that
\begin{equation}
 \tr W_N^\new = A(N,k) + \xi_0(N,k) + \xi_1(N,k),
\end{equation}
where $A(N,k)$ is the ``main term''
\[ A(N,k) = (-1)^{k/2} \sum_{Q | N_2} \mu(N_2/Q) b(N_1 Q^2) h'(-DQ^2) \]
and $\xi_0(N,1)$ and $\xi_1(N,1)$ are the correction terms given as
\begin{align*}
\xi_0(N,k) &=  \sum_{ Q | N_2 , \, N_1 Q^2 \le 4} \mu(N_2/Q)
 \left( \tr W_{N_1 Q^2}  - (-1)^{k/2} b(N_1 Q^2) h'(-DQ^2) \right), \\
 \xi_1(N,k) &= \delta_{k=2} \sum_{ Q | N_2, \, N_1 Q^2 > 4} \mu(N_2/Q).
\end{align*}

First we rewrite $A(N,k)$.
Note that $b(N_1 Q^2) = b(N)$ for all $Q \mid N_2$ unless
$-N_1$ is a fundamental discrimimant, $N_2$ is even, $Q$ is odd
and ${-N_1 \leg 2} \ne 1$.  In the latter instance, $b(N) = 1$ and
for $Q \mid N_2$, $b(N_1 Q^2)$ is 1 or 2 when $Q$ is even or odd, respectively.
Bearing this in mind, using \eqref{eq:altsum2} we can write \begin{equation} \label{eq:trnew3}
A(N,k) =
 (-1)^{k/2} b(N) c'(-D,N_2) h'(-D),
 \end{equation}
where
\[ 
c'(-D,N_2) =
\begin{cases}
 \frac 12 c(-D,N_2)
& N_1 \equiv 3 \mod 8, \, 2 \parallel N_2  \\
 c(-D,N_2) & \text{ else}. 
\end{cases} \]
The $\frac 12 c(-D, N_2)$ case arises as $c(-D, N_2) - c(-D, N_2^\odd)$,
and then observing that $c(-D, N_2) = 2 c(-D, N_2^\odd)$ when
$N_1 \equiv 3 \mod 8$ and $2 \parallel N_2$. 
Consequently, the function $c'(-D, N_2)$ also has
image in $\Z_{\ge 0}$ and is 0 precisely when $c(-D, N_2)$ is.
Note that in general we can write
\[ b(N) c'(-D,N_2) = \frac 12 \beta(N) c(-D,N_2) \]
where $\beta(N)$ is as in \eqref{eq:beta}.

Now we simplify the correction terms.

Terms with $N_1 Q^2 \le 4$ occur if and only if 
(i) $Q = 1$, $N_2$ is squarefree and
$N_1 \le 3$; or (ii) $Q=2$, $N_2$ is twice a squarefree number
and $N_1 = 1$. 
Hence, 
\[
\xi_0(N,k) = 
\begin{cases}
\mu(N_2) \left(  \tr W_1 - (-1)^{k/2} \cdot \frac 14 \right) & N_1 = 1, \, N_2 \text{ odd squarefree} \\
\mu(N_2)  \left(  \tr W_1 - \tr W_4 + (-1)^{k/2}\cdot \frac 14  \right) & N_1 = 1, \, N_2 \text{ even squarefree} \\ 
\mu(N_2/2)  \left(  \tr W_4 - (-1)^{k/2}\cdot \frac 12 \right) & N_1 = 1, \, N_2/4 \text{ odd squarefree} \\ 
\mu(N_2) \left( \tr W_2 -(-1)^{k/2} \cdot\frac 12 \right) & N_1 = 2, \, N_2 \text{ squarefree} \\
\mu(N_2) \left( \tr W_3 - (-1)^{k/2} \cdot\frac 23 \right) & N_1 = 3, \, N_2 \text{ squarefree}, \\
\end{cases} \]
and $\xi_0(N,k) = 0$ otherwise.

For the other correction term, note that
\[ \sum_{Q | N_2} \mu(N_2/Q) \cdot 1 = \sum_{Q | N_2} \mu(Q)  
= \delta_{N_2 = 1}. \]
Hence 
\[ \xi_1(N,k) = \delta_{k=2} \left( \delta_{N_2 = 1} + \eps(N) \right) \]
where $\eps(N) = -\sum_{ Q | N_2, \, N_1 Q^2 \le 4} \mu(N_2/Q)$.
Explicitly, we have
\[
\eps(N) = 
\begin{cases}
-\mu(N_2) & N_1 = 1,\, N_2 \text{ odd squarefree} \\
-\mu(N_2/2) & N_1 = 1,\, N_2 /4 \text{ odd squarefree} \\
-\mu(N_2) & N_1 = 2, 3, \, N_2 \text{ squarefree} \\
0 & \text{else}.
\end{cases}
\]

\subsection{Formulas for traces on newspaces}
Combining the explicit expressions for the main and correction
terms in the previous section gives the following.

\begin{prop} \label{prop:trnew}
Let $N \ge 1$, and let $k \ge 2$ be even.  Write
$N = N_1 N_2^2$ where $N_1$ is squarefree, and let
$-D \in \{ -N_1, -4N_1 \}$ be the discriminant of $\Q(\sqrt{-N})$.
\begin{enumerate}
\item
If $N$ is not $1$, $2$, $3$ or $4$ times the square of a squarefree number,
then
\[ \tr W_N^\new = \frac 12 (-1)^{k/2} \beta(N) c(-D, N_2) h'(-D) + \delta_{k=2} \delta_{N_2 = 1}. \]

\item
Suppose $N = N_2^2$, where $N_2$ is squarefree.  If $N_2$ is odd, then
\[  \tr W_N^\new = \frac 14 (-1)^{k/2} \left(c(-4, N_2) - \mu(N_2) \right) +
\mu(N_2) \left( \left\lfloor \frac k{12} \right\rfloor - \kappa_1^\odd \right) + \delta_{k=2}  \delta_{N_2 = 1}, \]
where $\kappa_1^\odd = 1$ if $k \equiv 2 \mod 12$ and $\kappa_1^\odd = 0$ otherwise.
If $N_2$ is even, then
\[  \tr W_N^\new = \frac 14 (-1)^{k/2} \left(c(-4, N_2) + \mu(N_2) \right) +
\mu(N_2) \left( \left\lfloor \frac k{12} \right\rfloor + \kappa_1^\even \right), \]
where $\kappa_1^\even = 1$ if $k \equiv 6, 10 \mod 12$ and $\kappa_1^\even = 0$ otherwise.

\item
If $N = N_2^2$, where $N_2$ is twice an even squarefree number, then
\[  \tr W_N^\new = \frac 14  (-1)^{k/2} c(-4, N_2) 
- \frac 12\mu(N_2/2). \]

\item
If $N = 2 N_2^2$, where $N_2$ is squarefree, then
\[  \tr W_N^\new = \frac 12 (-1)^{k/2}  \left(c(-8, N_2)  +
\kappa_2 \mu(N_2) \right) + \delta_{k=2}  \delta_{N_2 = 1}, \]
where $\kappa_2 = 1$ if $k \equiv 0, 2 \mod 8$ and $\kappa_2 = -1$
if $k \equiv 4, 6 \mod 8$.

\item
If $N = 3 N_2^2$, where $N_2$ is squarefree, then
\[\tr W_N^\new =  \frac 13(-1)^{k/2} \left(
\frac 12 \beta(N) c(-3, N_2) + \kappa_3 \mu(N_2)
\right) + \delta_{k=2}  \delta_{N_2 = 1},  \]
\end{enumerate}
where $\kappa_3 = -2$ if $k \equiv 4, 10 \mod 12$ and $\kappa_3
= 1$ otherwise.
\end{prop}

\begin{proof} From the previous section, we see that
\[ \tr W_N^\new = \frac 12 (-1)^{k/2} \beta(N) c(-D, N_2) h'(-D) 
+ \xi_0(N,k) + \delta_{k=2} (\delta_{N_2 =1} + \eps(N)). \]
In case (1), $\xi_0(N,k)  = \eps(N) = 0$, which yields the formula.

For the remaining cases, one just explicates 
$\xi_0(N,k) + \delta_{k=2} \eps(N)$ using the special cases for
$\tr W_N$ in \cref{sec:smallN}.  For instance, in case (2) with
$N_2$ odd, we have $-D=-4$ so $h'(-4) = \frac 12$ and
\begin{align*} \xi_0(N,k) + \delta_{k=2} \eps(N)
&= \mu(N_2) \left(\tr W_1 - (-1)^{k/2} \frac 14 - \delta_{k=2} \right) \\
&= \mu(N_2) 
\left( \left\lfloor \frac k{12} \right\rfloor - \kappa_1^\odd - (-1)^{k/2} \frac 14
\right).
\end{align*}
\end{proof}

\begin{cor} \label{cor:main}
 Keep the notation of \cref{prop:trnew}.
\begin{enumerate}
\item Suppose $N$ is not $1$, $2$, $3$ or $4$ times the square of a squarefree
number.  Then $\Delta(N,k) \ge 0$.
Furthermore, $\Delta(N,k) = 0$ exactly when
(i) $k=2$ and either $\dim S_k^\new(N) = 0$ or $N \in \{37, 58 \}$;  or
(ii) $2 \parallel N_2$ and $N_1 \equiv 7 \mod 8$.

\item Suppose $N = N_2^2$ where $N_2$ is squarefree.  Then
$\Delta(N,k) \ge 0$ for all $N$ if $k \le 10$ or $k=14$.
For any fixed $k$, $\Delta(N,k) \ge 0$ for $N$ sufficiently large.
For fixed $N$ and sufficiently large $k$, $|\Delta(N, k)| > 0$ and the sign of 
$\Delta(N,k)$ is $(-1)^{k/2} \mu(N_2)$.

\item Suppose $N = N_2^2$ where $N_2$ is twice an even
squarefree number.  Then $\Delta(N,k) \ge 0$.  Further, 
$\Delta(N,k) = 0$ if and only if $N=16$ and $k \equiv 2 \mod 4$. 

\item
Suppose $N = 2 N_2^2$, where $N_2$ is squarefree.  Then
$\Delta(N,k) \ge 0$.  Further $\Delta(N,k) = 0$ if and only if
 (i) $N \in \{ 8, 18 \}$ and $k \equiv 0, 2 \mod 8$; 
(ii) $N \in \{ 2, 72 \}$ and $k \equiv 4, 6 \mod 8$; or (iii) $(N,k) = (2,2)$.

\item
Suppose $N = 3 N_2^2$, where $N_2$ is squarefree.  Then
$\Delta(N,k) \ge 0$. Further $\Delta(N,k) = 0$ if and only if
(i) $N=\{ 3, 108 \}$ and $k \equiv 4, 10 \mod 12$; 
(iii) $N=12$ and $k \not \equiv 4, 10 \mod 12$; or (iii) $(N,k) = (3,2)$.
\end{enumerate}
\end{cor}

\begin{proof}
Recall that we always have $b(N) \ge \frac 12 $, $c(-D, N_2) \in \Z_{\ge 0}$ and 
$h'(-D) \ge \frac 13$.  Moreover $c(-D, N_2) = 0$ if and only if
either $2 \parallel N_2$ is even and ${-D \leg 2} = 1$.

Hence the statement that $\Delta(N, k) \ge 0$ in (1) is obvious when $k \ne 2$.  When $k = 2$, the only possible negative
term in our formula for $(-1)^{k/2} \tr W_N^\new$ is $-\delta_{N_2 = 1}$.  However this
only occurs when $N$ is squarefree, in which case 
 we already know $\Delta(N_1, 2) \ge 0$ with equality if and only if
$\dim S_2^\new(N) = 0$ or  $N \in \{ 37, 58 \}$  by \cite{me:refdim}.  
 When $N$ is not squarefree, $\Delta(N,k) = 0$ if and only if
 $c(-D,N_2) = 0$.  This finishes (1).
 
 Suppose we are in case (2) now.  Note $c(-4,N_2) = \prod_{p | N_2^\odd}
 (p - 1 - {-4 \leg p}) \ge 3^{\omega(N_2^\odd)}$.  
 Here $\omega(n)$ denotes the number of prime divisors of $n$.
For fixed $k$, $\Delta(N,k)$ is thus dominated by $c(-4,N_2)$
as $N_2 \to \infty$.  For fixed $N$, $\Delta(N,k)$ is dominated by
$(-1)^{k/2} \mu(N_2) \lfloor k/12 \rfloor$ as $k \to \infty$.
If $k < 12$ and $k=14$, the only possible negative terms in our
formulas for $(-1)^{k/2} \tr W_N^\new$ will be a $-1$ (depending on 
$\mu(N_2)$ and $k$).  However, in these cases is also a strictly 
positive term, and since the trace is an integer, 
$(-1)^{k/2} \tr W_N^\new \ge 0$.
 
 In case (3), we have $c(-4, N_2) = 2 \prod_{p | N_2^\odd} (p - 1 - {-1 \leg p})
 \ge 2$, with equality if and only if $N = 16$.  This easily gives the assertions
 in (3).
 
 Consider case (4).  When $N_2 = 1$, by \cite{me:refdim}
$\Delta(2,k) \ge 0$ with equality if and only if $k \equiv 4, 6 \mod 8$. 
Suppose $N_2 > 1$.  Then $\Delta(N,k) \ge 0$ if and only if
$c(-8,N_2) \ge - \kappa_2 \mu(N_2)$.
Note $c(-8, N_2) = \prod_{p | N_2} (p-1 - {-8 \leg p}) \ge 1$
and equals $1$ if and only if $N_2 | 6$.  Then the assertions readily follow.
 
Finally consider case (5).  Again, the squarefree case $N=3$ is treated
in \cite{me:refdim}, so suppose $N_2 > 1$.  Now $\beta(N) c(-3,N_2)$
equals $\prod_{p | N_2} (p - 1 - {-3 \leg p}) \ge 2^{\omega(N_2)}$ if $N_2$ is odd
and $\frac 12 \prod_{p | N_2^\odd} (p - 1 - {-3 \leg p}) \ge 2^{\omega(N_2^\odd-1)}$ if $N_2$ is even.  Comparing with the proposition now finishes the proof.
\end{proof}

This proves the theorems stated in the introduction.


\section{Excluding twists for certain levels} \label{sec:p2}

Motivated by the different behavior in root number bias for
cubefree square levels, we briefly investigate what happens
when we exclude certain twists from smaller levels.  
Our main goal is to prove the propositions in \cref{sec:12}.

First we recall some facts about twists of modular forms
of cube-free levels.
We will explain things from the point of view of associated
local representations, as this 
perspective makes things more transparent to us.  However, many of
the facts that we recall are also well known from the
more classical ``global'' approach (see
 \cite{AL}, \cite{atkin-li}, \cite{HPS:twists}).
 
Given a newform $f \in S_k^\new(N,\psi)$ with nebentypus
$\psi$, there is an associated cuspidal automorphic 
representation $\pi = \pi_f$ of $\GL_2(\A_{\mathbb Q})$.
We have a decomposition 
$\pi = \bigotimes_p \pi_p \otimes \pi_\infty$, where each
$\pi_p$ is an infinite-dimensional irreducible admissible 
complex representation of $\GL_2(\Q_p)$.  We say
$\pi_p$ is the local representation at $p$ associated to $f$.
The local conductor $c(\pi_p)$ is the exponent $v_p(N)$
in the level of $f$.

Assume now that $\psi$ is trivial, i.e., $f \in S_k^\new(N)$.
Then $\pi$, and each $\pi_p$, has trivial central character.
Moreover, $\pi_p$ is an unramified principal series representation
(i.e., $c(\pi_p) = 0$) if and only if $p \nmid N$.  
Also, $\pi_p$ is an unramified
quadratic twist (possibly trivial twist) of the Steinberg representation
$\mathrm{St}_p$ of $\GL_2(\Q_p)$ (i.e., $c(\pi_p) = 1$) 
if and only if $p \parallel N$.

Now there are 3 possibilities when $c(\pi_p) = 2$:
$\pi_p$ can be a ramified principal series, a ramified quadratic
twist of Steinberg, or supercuspidal.  However, the former two
possibilities do not happen for $p=2$ (e.g., see 
\cite[Corollary 4.1]{pacetti}).  In fact there are 2 possibilities
for $\pi_p$ being a ramified principal series: either it is a ramified
quadratic twist of an unramified principal series, or a minimal twist (a twist
minimizing the local conductor) of $\pi_p$ is a ramified principal series
of conductor 1 (necessarily the central character is nontrivial
and has conductor 1).
In any event, only the supercuspidal representations are minimal.

Suppose now that $N=M^2$ where $M > 1$ is squarefree.
Then the above description of local representations means exactly 
one of the following is true for a newform $f \in S_k^\new(N)$: 

\begin{enumerate}
\item
$f$ is minimal, i.e., no twist (possibly with nebentypus)
has smaller level.  Here $\pi_p$ is supercuspidal for each $p \mid N$.

\item
$f$ is a quadratic twist of a level 1 newform.  Here $\pi_p$
is a ramified quadratic twist of an unramified principal series for each $p 
\mid N$.  Necessarily $N$ is odd.

\item
A minimal twist $f'$ of $f$ (possibly with nebentypus) has level $N'$ strictly between $1$ and $N$.  Necessarily $v_2(N') = v_2(N)$.
\end{enumerate}

The utility of the local representation theory perspective for us
comes both from a straightforward description of the level of twists
and the fact that root numbers can be read off of the local representations.
Namely, the root number of $f$ is the product over $p$ of the local
representation root numbers $w_p(\pi_p)$
times $(-1)^{k/2}$ (which is the local root number of $\pi_\infty$).

Recall $S_k^\new(N)'$ (resp.\ $S_k^{\min}(N)$) is defined to be the subspace of $S_k^\new(N)$ generated by newforms of types (1) and (3) above
(resp.\ of type (1) above).  Define $\tr W_N'$ and $\tr W_N^{\min}$
to be the traces of $W_N$ restricted to the spaces $S_k^\new(N)'$
and $S_k^{\min}(N)$, respectively.  First we describe $\tr W_N'$.

\begin{lem} \label{lem:41} We have
\[ \tr W_N' = 
\begin{cases}
\tr W_N^\new - (-1)^{k/2} \prod_{p | N} {-1 \leg p} \cdot \dim S_k(1) & N \text{ odd} \\
\tr W_N^\new & N \text{ even}. 
\end{cases} \]
\end{lem}

\begin{proof}
Since $S_k^\new(N)' = S_k^\new(N)$ if $N$ is even, assume $N$ is odd.
Then we simply need to subtract off from $\tr W_N^\new$ 
the trace of $W_N$ restricted
to the subspace of $S_k^\new(N)$ generated by forms of type (2).
Now we simply use the fact that if $\pi_p$ is a ramified
quadratic twist of an unramified principal series, then $w_p(\pi_p) =
{-1 \leg p}$ (see \cite[Theorem 3.2]{pacetti} in terms of local
representations, or \cite[Theorem 6]{AL} in terms of modular forms).
\end{proof}

Let $\Delta(N,k)' = (-1)^{k/2} \tr W_N'$, which is the difference between
the number of newforms in $S_k^\new(N)'$ with root number $+1$
and $-1$.

\begin{proof}[Proof of \cref{prop:13}]
When $N$ is even, we simply have $\Delta(N,k)' = \Delta(N,k)$,
so assume $N$ is odd.  Then \cref{prop:trnew} combined with the
above lemma shows $\Delta(N,k)'$ is
\[ \frac 14(c(-4,M)-\mu(M)) -
\left( \prod_{p | M} {-1 \leg p} - (-1)^{k/2} \mu(M) \right) \dim S_k(1)
 + \mu(M) \delta_{k=2}. \]
 The sum of the first and the last term is always non-negative.
Note the coefficient of $\dim S_k(1)$ is negative if and only if
$\mu(M) = -(-1)^{k/2}$ and $\prod_{p | M} {-1 \leg p} = 1$.
In this case, $\Delta(N,k)' < 0$ if $k$ is sufficiently large.
\end{proof}

In general to describe $\tr W_N^{\min}$, one needs to isolate
the forms which are supercuspidal at each ramified place.
One can presumably do this in a similar way as \cref{lem:41}, by subtracting
off the contributions from twists from smaller level.  However,
we will restrict ourselves to giving a couple of examples
to show that $(-1)^{k/2} \tr W_N^{\min}$ can be negative, and then
prove that often there is a quadratic twist which forces 
$\tr W_N^{\min} = 0$ to get \cref{prop:14}.

The example of $N=4$ was already given in \cref{sec:12}.
The following examples are taken from the LMFDB \cite{LMFDB}.

\begin{ex} The space $S^\new_{10}(9)$ has 3 newforms, all
rational.  Two are twists from $S^\new_{10}(3)$ by the quadratic 
character ${ -3 \leg \cdot}$, and they have root number $+1$ 
(though the corresponding forms on 
$S^\new_{10}(3)$ have opposite root numbers).
The other newform in $S^\new_{10}(9)$, which is CM and minimal,  
has root number $-1$.  Hence the sum of root numbers
of minimal newforms in $S_{10}(9)$ is $-1$, though
$\Delta(10,9) = +1$.
\end{ex}



\begin{ex} The space $S_{14}^{\min}(49)$ has 3 Galois orbits of
newforms: one CM orbit of size 1 and root number $+1$,
one Galois orbit of size 6 and root number $+1$, and one Galois
orbit of size 12 and root number $-1$.  Hence the sum of root
numbers for minimal forms is $1 \cdot 1 + 1 \cdot 6 - 1 \cdot 12 = -5$.
On the other hand, $\Delta(49,14) = 2$. 
We remark that except for
the first form, all newforms are non-CM, and all of the newforms
have nontrivial inner twists.  
\end{ex}

\begin{proof}[Proof of \cref{prop:14}]
Here we allow any $N > 1$, i.e., $N$ need not be a square or cubefree.  
Suppose $p^2 \parallel N^\odd$, and let $\chi$
be the quadratic character of conductor $\pm p$.  Then twisting any
minimal newform $f \in S_k^{\min}(N)$ gives another newform of level
$N$.  The local argument for this is that twisting a supercuspidal 
representation $\pi_p$ with $c(\pi_p) = 2$ by $\chi_p$
yields a supercuspidal representation $\pi_p \otimes \chi_p$
also of local conductor 2.

Hence twisting by $\chi$ is an involution on the set of minimal newforms.
Now any supercuspidal representation $\pi_p$ with $c(\pi_p)=2$
is necessarily minimal, and thus by \cite[Proposition 3.5]{tunnell},
is induced from the unramified quadratic extension of $\Q_p$.
Hence by \cite[Theorem 3.2]{pacetti}, the local root number of 
$\pi_p \otimes \chi_p$ is $-{-1 \leg p}$ times the local root number of
$\pi_p$.  Moreover, for $q \mid p^{-2} N$, twisting by $\chi_p$ multiplies
the root number of $\pi_q$ by ${q \leg p}^{v_q(N)}$.
  In particular, if ${-N/p^2 \leg p} = 1$, then twisting by $\chi$
flips the sign of the root number of minimal newforms, which implies
$\tr W_N^{\min} = 0$.
\end{proof}

%
%


\begin{bibdiv}
\begin{biblist}

\bib{AL}{article}{
   author={Atkin, A. O. L.},
   author={Lehner, J.},
   title={Hecke operators on $\Gamma _{0}(m)$},
   journal={Math. Ann.},
   volume={185},
   date={1970},
   pages={134--160},
   issn={0025-5831},
}

\bib{atkin-li}{article}{
   author={Atkin, A. O. L.},
   author={Li, Wen Ch'ing Winnie},
   title={Twists of newforms and pseudo-eigenvalues of $W$-operators},
   journal={Invent. Math.},
   volume={48},
   date={1978},
   number={3},
   pages={221--243},
   issn={0020-9910},
}


\bib{cohn}{book}{
   author={Cohn, Harvey},
   title={A classical invitation to algebraic numbers and class fields},
   series={Universitext},
   note={With two appendices by Olga Taussky: ``Artin's 1932 G\"{o}ttingen
   lectures on class field theory'' and ``Connections between algebraic
   number theory and integral matrices''},
   publisher={Springer-Verlag, New York-Heidelberg},
   date={1978},
   pages={xiii+328},
   isbn={0-387-90345-3},
}

\bib{HPS}{article}{
   author={Hijikata, Hiroaki},
   author={Pizer, Arnold K.},
   author={Shemanske, Thomas R.},
   title={The basis problem for modular forms on $\Gamma_0(N)$},
   journal={Mem. Amer. Math. Soc.},
   volume={82},
   date={1989},
   number={418},
   pages={vi+159},
   issn={0065-9266},
}

\bib{HPS:twists}{article}{
   author={Hijikata, Hiroaki},
   author={Pizer, Arnold K.},
   author={Shemanske, Thomas R.},
   title={Twists of newforms},
   journal={J. Number Theory},
   volume={35},
   date={1990},
   number={3},
   pages={287--324},
   issn={0022-314X},
}

\bib{LMFDB}{misc}{
  label    = {LMFDB},
  author       = {The {LMFDB Collaboration}},
  title        = {The {L}-functions and Modular Forms Database},
  note = {\url{http://www.lmfdb.org}},
  year         = {2022},
}

\bib{LPW}{article}{
   author={Luo, Zhilin},
   author={Pi, Qinghua},
   author={Wu, Han},
   title={Bias of root numbers for Hilbert newforms of cubic level},
   journal={J. Number Theory},
   volume={243},
   date={2023},
   pages={62--116},
   issn={0022-314X},
}

\bib{martin}{article}{
   author={Martin, Greg},
   title={Dimensions of the spaces of cusp forms and newforms on $\Gamma_0(N)$ and $\Gamma_1(N)$},
   journal={J. Number Theory},
   volume={112},
   date={2005},
   number={2},
   pages={298--331},
   issn={0022-314X},
}

\bib{me:refdim}{article}{
   author={Martin, Kimball},
   title={Refined dimensions of cusp forms, and equidistribution and bias of
   signs},
   journal={J. Number Theory},
   volume={188},
   date={2018},
   pages={1--17},
   issn={0022-314X},
}

\bib{me:cong2}{article}{
   author={Martin, Kimball},
   title={Congruences for modular forms ${\rm mod}\,2$ and quaternionic
   $S$-ideal classes},
   journal={Canad. J. Math.},
   volume={70},
   date={2018},
   number={5},
   pages={1076--1095},
   issn={0008-414X},
}

\bib{me:basis}{article}{
   author={Martin, Kimball},
   title={The basis problem revisited},
   journal={Trans. Amer. Math. Soc.},
   volume={373},
   date={2020},
   number={7},
   pages={4523--4559},
   issn={0002-9947},
}

\bib{pacetti}{article}{
   author={Pacetti, Ariel},
   title={On the change of root numbers under twisting and applications},
   journal={Proc. Amer. Math. Soc.},
   volume={141},
   date={2013},
   number={8},
   pages={2615--2628},
   issn={0002-9939},
}

\bib{PQ}{article}{
   author={Pi, Qinghua},
   author={Qi, Zhi},
   title={Bias of root numbers for modular newforms of cubic level},
   journal={Proc. Amer. Math. Soc.},
   volume={149},
   date={2021},
   number={12},
   pages={5035--5047},
   issn={0002-9939},
}

\bib{sage}{manual}{
      author={Developers, The~Sage},
       title={{S}agemath, the {S}age {M}athematics {S}oftware {S}ystem
  ({V}ersion 9.4)},
        date={2021},
        label={Sage},
        note={{\tt https://www.sagemath.org}},
}

\bib{SZ}{article}{
   author={Skoruppa, Nils-Peter},
   author={Zagier, Don},
   title={Jacobi forms and a certain space of modular forms},
   journal={Invent. Math.},
   volume={94},
   date={1988},
   number={1},
   pages={113--146},
   issn={0020-9910},
}

\bib{tunnell}{article}{
   author={Tunnell, Jerrold B.},
   title={On the local Langlands conjecture for $GL(2)$},
   journal={Invent. Math.},
   volume={46},
   date={1978},
   number={2},
   pages={179--200},
   issn={0020-9910},
}

\bib{yamauchi}{article}{
   author={Yamauchi, Masatoshi},
   title={On the traces of Hecke operators for a normalizer of $\Gamma
   _{0}(N)$},
   journal={J. Math. Kyoto Univ.},
   volume={13},
   date={1973},
   pages={403--411},
   issn={0023-608X},
}

\bib{ZZ}{article}{
   author={Zhang, Yichao},
   author={Zhou, Yang},
   title={Dimension formulas for modular form spaces with character for
   Fricke groups},
   journal={Acta Arith.},
   volume={206},
   date={2022},
   number={4},
   pages={291--311},
   issn={0065-1036},
}

\end{biblist}
\end{bibdiv}
\end{document}